\documentclass[11pt]{article}
\usepackage[margin=1in]{geometry}
\usepackage{amsmath,amssymb,amsthm,booktabs,tikz,wrapfig,microtype,capt-of}
\usepackage[hidelinks]{hyperref}
\hypersetup{pdftitle={The Hibi\textendash Li (2016) conjecture is fully resolved},pdfauthor={Dave Goel}}

\newtheorem{theorem}{Theorem}[section]
\newtheorem{lemma}[theorem]{Lemma}
\newtheorem{proposition}[theorem]{Proposition}
\newtheorem{corollary}[theorem]{Corollary}
\newtheorem{conjecture}[theorem]{Conjecture}
\theoremstyle{definition}
\newtheorem{definition}[theorem]{Definition}
\newtheorem{example}[theorem]{Example}
\theoremstyle{remark}

\newcommand{\OO}{\mathcal O}
\newcommand{\CC}{\mathcal C}
\newcommand{\SSS}{\mathcal S}
\newcommand{\TT}{\mathcal T}
\newcommand{\RR}{\mathcal R}
\newcommand{\relint}{\operatorname{relint}}
\newcommand{\aff}{\operatorname{aff}}

\newcommand{\pyr}{\operatorname{pyr}}

\title{The Hibi--Li (2016) conjecture is fully resolved}
\author{Dave Goel\thanks{The Lean~4 proof of this paper's results:
\url{https://github.com/david44hello/papers_supporting_materials/tree/hibi-li-2016-conjecture-v4.60/hibi-li-2016-conjecture}.}}
\date{\small November 2026}

\begin{document}
\maketitle

\begin{abstract}
Hibi and Li conjectured that the order polytope $\OO(P)$ of a finite
poset $P$ has at most as many faces of each dimension as its chain
polytope $\CC(P)$, and that equality in some dimension $k$ with
$1\leq k\leq|P|-1$ forces the two polytopes to be unimodularly
equivalent. Fourier proved the
inequality in a 2026 preprint. Let $X$ be the five-element poset in
which an element lies above two incomparable elements and below
two other incomparable elements. Using two results from Fourier's preprint, we show that
if an $n$-element poset contains $X$ as an induced subposet, then, for
every $k$ with $2\leq k\leq n-1$, its chain polytope has at least
$\binom{n-3}{k-2}$ more $k$-dimensional faces than its order polytope,
and that for every $n\geq5$ some $n$-element poset containing $X$ attains all these
bounds at once. Since Hibi and Li showed that $\OO(P)$ and $\CC(P)$
are unimodularly equivalent exactly when $P$ does not contain $X$, this proves the equality statement for these values of $k$. The second version of Fourier's preprint
gives an independent proof of this statement. For $k=1$, the equality statement is false: the polytopes $\OO(X)$ and $\CC(X)$ have the same number of edges (Hibi, Li, Sahara and
Shikama) but are not unimodularly equivalent.
Together with Fourier's inequality, these results settle the conjecture. The
proofs given in this paper are formally verified in Lean~4.
\end{abstract}

{\small\noindent\textit{Keywords:} order polytope, chain polytope, face numbers, unimodular equivalence,
Hibi--Li conjecture.

\par}

\section{Introduction}\label{sec:intro}
This paper completes the resolution of the Hibi--Li (2016) conjecture
(Conjecture~\ref{conj:HL}). Let $P$ be a finite poset with $n$
elements. Its order polytope $\OO(P)$ and chain polytope $\CC(P)$,
both due to Stanley and defined in
Section~\ref{sec:defs}, are $n$-dimensional $0/1$ polytopes with
the same number of vertices \cite[Corollary~1.3 and Theorem~2.2]{Stanley}. For a polytope $Q$, let
$f_k(Q)$ be the number of \emph{$k$-faces} of $Q$, that is, of faces of dimension $k$.

\begin{conjecture}[{\cite[Conjecture~2.4, p.~12]{HL}}]\label{conj:HL}
Let $P$ be a finite poset with $|P|=d>1$. Let
$f(\OO(P))=(f_0,f_1,\ldots,f_{d-1})$ denote the $f$-vector of
$\OO(P)$ and $f(\CC(P))=(f_0',f_1',\ldots,f_{d-1}')$ the $f$-vector of
$\CC(P)$. Then
\begin{enumerate}
\item[(a)] $f_i\leq f_i'$ for all $1\leq i\leq d-1$.
\item[(b)] If $f_i=f_i'$ for some $1\leq i\leq d-1$, then $\OO(P)$ and
$\CC(P)$ are unimodularly equivalent.
\end{enumerate}
\end{conjecture}

\begin{figure}[htb]
\centering
\begin{minipage}{0.40\textwidth}
\centering
\begin{tikzpicture}[scale=0.6,baseline=(current bounding box.north)]
 \foreach \nm/\x/\y/\side in {a/0/0/left,b/2/0/right,c/1/1/right,g/0/2/left,h/2/2/right}
   {\fill (\x,\y) circle (2.2pt); \node[\side] at (\x,\y) {\small $\nm$};}
 \draw (0,0)--(1,1)--(2,0); \draw (0,2)--(1,1)--(2,2);
\end{tikzpicture}
\qquad\qquad
\begin{tikzpicture}[scale=0.6,baseline=(current bounding box.north)]
 \foreach \nm/\x/\y/\side in {a/0/0/left,b/2/0/right,c/1/1/right,g/0/2/left,h/2/2/right}
   {\fill (\x,\y) circle (2.2pt); \node[\side] at (\x,\y) {\small $\nm$};}
 \foreach \y in {-1,-2} {\fill (1,\y) circle (2.2pt);}
 \draw (0,0)--(1,1)--(2,0); \draw (0,2)--(1,1)--(2,2);
 \draw (0,0)--(1,-1)--(2,0); \draw (1,-1)--(1,-2);
\end{tikzpicture}
\end{minipage}\hfill
\begin{minipage}{0.56\textwidth}
\caption{The poset $X$ (left) and the poset $P_7$ of Proposition~\ref{prop:sharp} (right). The poset $P_7$ is obtained from $X$ by adjoining a new least element two times in succession, and it attains the bound of Theorem~\ref{thm:main}.}\label{fig:X}
\end{minipage}
\end{figure}
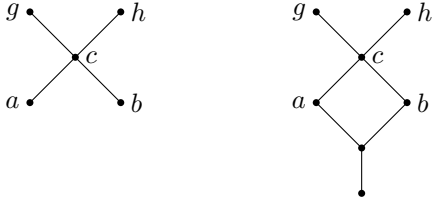
Conjecture~\ref{conj:HL} is quoted verbatim from Hibi and Li \cite{HL}. Below we write $n$
for $d$ and $k$ for $i$. Unimodular equivalence is defined in Section~\ref{sec:defs}. Let $X$ be the poset of Figure~\ref{fig:X} (left), whose elements satisfy $a,b<c<g,h$, where $a$ and $b$ are incomparable, and so are $g$ and $h$. We say that $P$ \emph{contains} $X$ if five
elements of $P$ satisfy these relations and incomparabilities, that is, if $P$ has an induced subposet isomorphic to $X$, as in \cite[Theorem~2.1]{HL}; the relations need not be cover relations. Part (a) of Conjecture~\ref{conj:HL} was proved by Fourier in a 2026 preprint \cite[Theorem~1.1]{Fourier}; it had been proved earlier for maximal ranked posets \cite[Corollary~4.2]{AFJ} and for posets that can be built by disjoint unions and ordinal sums from posets that do not contain $X$ \cite[Theorem~1]{FHL}.

Part (b) is false for $k=1$: the polytopes $\OO(P)$ and $\CC(P)$ always have the same number of edges \cite[Theorem~2.1]{HLSS}, but
$\OO(X)$ and $\CC(X)$ are not unimodularly equivalent, as they have $8$
and $9$ facets, respectively (Example~\ref{ex:X}). Later papers restate part
(b) for $k\geq2$ \cite[Conjecture~1.1]{FLM}, \cite[Conjecture~28]{Lee}.
We prove part (b) for all $k$ with $2\leq k\leq n-1$.
The second version of Fourier's preprint gives an independent proof of this
statement \cite[Theorem~8.2]{FourierV2}; it shows that, for posets that
contain $X$, the inequalities of part (a) are strict for $2\leq k\leq n-1$,
and it states no explicit lower bound for the differences
$f_k(\CC(P))-f_k(\OO(P))$.

For a polytope
$Q$, let $F_Q(t)=\sum_{j\geq0}f_j(Q)t^j$ be its \emph{face polynomial};
it counts all nonempty faces, $Q$ included. Our first result, Theorem~\ref{thm:main}, is a sharp lower bound: if
$P$ contains $X$, then $F_{\CC(P)}-F_{\OO(P)}\geq t^2(1+t)^{n-3}$
coefficientwise, so $\CC(P)$ has at least $\binom{n-3}{k-2}$ more
$k$-faces than $\OO(P)$ for $2\leq k\leq n-1$. For every $n\geq5$, some $n$-element poset containing $X$ attains all these bounds. Our second result,
Corollary~\ref{cor:equality}, settles the equality case: for
$2\leq k\leq n-1$ we have $f_k(\OO(P))=f_k(\CC(P))$ if and only if $P$
does not contain $X$, that is, by Hibi and Li's criterion (Theorem~\ref{thm:HL}), if and only if $\OO(P)$ and $\CC(P)$ are unimodularly equivalent. Together with part (a) and the case $k=1$, this resolves Conjecture~\ref{conj:HL}:

\begin{theorem}\label{thm:resolution}
Let $P$ be a finite poset with $n\geq2$ elements.
\begin{enumerate}
\item[(a)] \textup{(Fourier \cite[Theorem~1.1]{Fourier})} We have $f_k(\OO(P))\leq f_k(\CC(P))$ for every $k$.
\item[(b)] For $2\leq k\leq n-1$ we have $f_k(\OO(P))=f_k(\CC(P))$ if and only if $\OO(P)$ and $\CC(P)$ are
unimodularly equivalent \textup{(Corollary~\ref{cor:equality})}.
\item[(c)] We have $f_1(\OO(P))=f_1(\CC(P))$ \textup{\cite[Theorem~2.1]{HLSS}}, while $\OO(P)$ and $\CC(P)$ are not unimodularly equivalent if $P$ contains $X$ \textup{(Theorem~\ref{thm:HL}; for $P=X$, see also Example~\ref{ex:X})}.
So, for every poset that contains $X$, part (b) of Conjecture~\ref{conj:HL} fails for $k=1$.
\end{enumerate}
\end{theorem}
Table~\ref{tab:compare} compares Theorem~\ref{thm:main} and Corollary~\ref{cor:equality} with earlier work, and Figure~\ref{fig:beforeafter} shows the equality part before and after this paper.\usetikzlibrary{patterns}%
\begin{figure}[htb]
\centering
\begin{tikzpicture}[x=0.62cm,y=0.28cm,
  known/.style={circle,fill=black,inner sep=1.5pt},
  cross/.style={font=\footnotesize},
  num/.style={font=\scriptsize,inner sep=0.6pt}]
 \begin{scope}
  \fill[pattern=north east lines,pattern color=black!45]
    (-0.45,1.55) -- (4.45,1.55) -- (4.45,6.95) -- (-0.45,2.05) -- cycle;
  \foreach \n in {5,...,9} {
    \pgfmathtruncatemacro{\x}{\n-5}
    \node[cross,text=red!75!black] at (\x,0) {$\times$};
    \node[known] at (\x,1) {};
    \pgfmathtruncatemacro{\top}{\n-2}
    \node[known] at (\x,\top) {};
    \node[font=\scriptsize] at (\x,-1) {$\n$};
  }
  \foreach \n in {6,...,9} {
    \pgfmathtruncatemacro{\x}{\n-5}
    \pgfmathtruncatemacro{\last}{\n-2}
    \foreach \k in {3,...,\last} {
      \pgfmathtruncatemacro{\y}{\k-1}
      \node[font=\scriptsize,fill=white,inner sep=0.5pt] at (\x,\y) {?};
    }
  }
  \node[font=\scriptsize,fill=white,inner sep=0.5pt] at (0,2) {?};
  \foreach \k in {1,...,8} {\pgfmathtruncatemacro{\y}{\k-1}\node[font=\scriptsize,anchor=east] at (-0.55,\y) {$\k$};}
  \node[font=\scriptsize,anchor=east] at (-1.1,3.5) {$k$};
  \node[font=\scriptsize,anchor=east] at (-0.55,-1) {$n$};
  \node[font=\small,anchor=south] at (2,8.1) {known before this paper};
  \node[font=\scriptsize,anchor=west] at (4.5,0) {\textcolor{red!75!black}{false} \cite{HLSS,HL}};
  \node[font=\scriptsize,anchor=west] at (4.5,1) {known \cite{FLM}};
  \node[font=\scriptsize,anchor=west] at (4.5,7) {known \cite{HL}};
  \node[font=\scriptsize,anchor=west] at (4.5,4) {open \cite{Lee}};
 \end{scope}
 \begin{scope}[xshift=7.4cm]
  \fill[green!22] (-0.45,0.55) -- (4.45,0.55) -- (4.45,7.95) -- (-0.45,3.05) -- cycle;
  \foreach \n in {5,...,9} {
    \pgfmathtruncatemacro{\x}{\n-5}
    \node[cross,text=red!75!black] at (\x,0) {$\times$};
    \pgfmathtruncatemacro{\last}{\n-1}
    \foreach \k in {2,...,\last} {
      \pgfmathtruncatemacro{\y}{\k-1}
      \pgfmathtruncatemacro{\b}{round(factorial(\n-3)/(factorial(\k-2)*factorial(\n-1-\k)))}
      \node[num] at (\x,\y) {$\b$};
    }
    \node[font=\scriptsize] at (\x,-1) {$\n$};
  }
  \foreach \k in {1,...,8} {\pgfmathtruncatemacro{\y}{\k-1}\node[font=\scriptsize,anchor=east] at (-0.55,\y) {$\k$};}
  \node[font=\scriptsize,anchor=east] at (-1.1,3.5) {$k$};
  \node[font=\scriptsize,anchor=east] at (-0.55,-1) {$n$};
  \node[font=\small,anchor=south] at (2,8.1) {this paper};
  \node[font=\scriptsize,anchor=west] at (4.5,0) {\textcolor{red!75!black}{false}};
  \node[font=\scriptsize,anchor=west,align=left] at (4.5,4) {exact\\ minimum};
 \end{scope}
\end{tikzpicture}
\caption{The equality part of Conjecture~\ref{conj:HL} for posets with $n$ elements (columns) and faces of dimension $k$ (rows); for $n\leq4$ no poset contains $X$. Left, before this paper (sources in Table~\ref{tab:compare}): false for $k=1$ ($\times$), known for $k=2$ and for $k=n-1$ (dots), open otherwise (hatched, ?). Right, this paper: if $P$ contains $X$, then $\CC(P)$ has at least $\binom{n-3}{k-2}$ more $k$-dimensional faces than $\OO(P)$ (the number at each point), and one $n$-element poset attains all these bounds at once (Theorem~\ref{thm:main}); so, for $2\leq k\leq n-1$, equality holds exactly when $P$ does not contain $X$ (Corollary~\ref{cor:equality}). The second version of Fourier's preprint proves this equality statement independently \cite[Theorem~8.2]{FourierV2}, without an explicit lower bound.}\label{fig:beforeafter}
\end{figure}
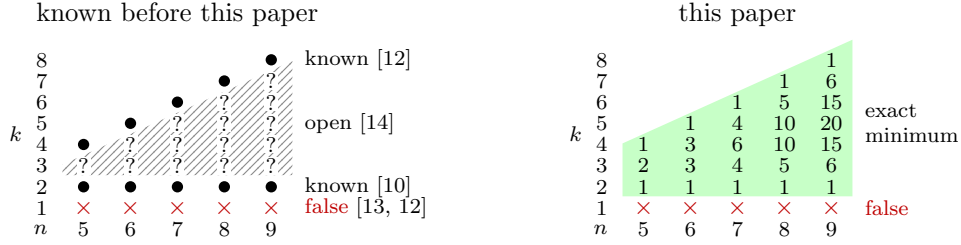

\begin{table}[htb]
\centering
\small
\caption{Comparison with earlier results.}\label{tab:compare}
\vspace{2pt}
\begin{tabular}{@{}p{0.42\textwidth}@{\hspace{0.03\textwidth}}p{0.55\textwidth}@{}}
\toprule
New result & Best result before this paper\\
\midrule
$f_k(\CC(P))-f_k(\OO(P))\geq\binom{n-3}{k-2}$ for $2\leq k\leq n-1$
when $P$ contains $X$; attained for every $n\geq5$
(Theorem~\ref{thm:main}).
& $f_k(\CC(P))-f_k(\OO(P))\geq0$ \cite[Theorem~1.1]{Fourier}; for the
difference of the numbers of facets, bounds \cite[Theorem~10]{Bha} and a formula \cite[Theorem~2.1]{FourierGap}.\\
\addlinespace
Part (b) of Conjecture~\ref{conj:HL} for \mbox{$2\leq k\leq n-1$}
(Corollary~\ref{cor:equality}); proved independently by Fourier
\cite[Theorem~8.2]{FourierV2}.
& Known for $k=2$ \cite[Theorem~6.1]{FLM} and for $k=n-1$
\cite[Theorem~1.3]{HL}; listed as open for $3\leq k\leq n-2$ in
\cite[Section~5]{Lee}; false for $k=1$ \cite[Theorem~2.1]{HLSS}, \cite[Theorem~2.1]{HL}.\\
\bottomrule
\end{tabular}
\end{table}

\paragraph{Outline of the proof.}
Fourier \cite{Fourier} passes from $\OO(P)$ to $\CC(P)$ in $n$ steps,
called switches; each switch changes only the inequalities that involve
a given coordinate $z$. The vector $u$ of the other coordinates ranges over the same polytope $B$
of dimension $n-1$ before and after the switch, and there are functions $L$ and $U$ on $B$ with $L\leq U$ such that $z$ ranges over $[L(u),U(u)]$ before the
switch and over $[0,U(u)-L(u)]$ after it. By Lemma~\ref{lem:identity}, which rests on an identity of Fourier, the change of the face polynomial is a sum of local terms, one for each face of $B$. These terms are computed from two subdivisions of $B$, one into the maximal cells on which $L$ is affine and one into those on which $U$ is affine, and from their common refinement. The local terms are nonnegative because the two subdivisions are transverse (Proposition~\ref{prop:transverse}, due to Fourier). Call a common facet of two cells of the same subdivision a \emph{wall}. For a poset that contains $X$, we find a switch and a $4$-dimensional face $E_0$ of $B$ in which a wall of the first subdivision crosses a wall of the second (Figure~\ref{fig:slice}). The crossing persists in every $e$-face
$E\supseteq E_0$ of $B$ and makes the local term of $E$ at least
$t^{e-2}(1+t)^2=t^{e-2}+2t^{e-1}+t^e$: the $(e-2)$-cell where the two walls cross
contributes $t^{e-2}$; the two walls, each cut in two by the other, contribute $2t^{e-1}$; and the four $e$-cells around the crossing, which form a cycle in the graph of Lemma~\ref{lem:carrier}(ii), contribute $t^e$. As $E_0$ lies in at least
$\binom{n-5}{e-4}$ such faces for each $e$, summing over $e$ gives at least
$t^2(1+t)^{n-3}$. Repeatedly adjoining a new least element to $X$ (Figure~\ref{fig:X}, right) gives posets whose order and chain polytopes are iterated pyramids over $\OO(X)$ and $\CC(X)$, respectively; each pyramid step multiplies
the difference by $1+t$, so these posets attain the bound.

We use without proof two results from Fourier's preprint \cite{Fourier}: the identity that writes the change
of the face polynomial at a switch through the face polynomials of $B$ and of the three subdivisions (its Lemma~4.1;
the first equality of our Lemma~\ref{lem:identity}), and the transversality of the two subdivisions of $B$ (its Proposition~5.2;
our Proposition~\ref{prop:transverse}).
We also use the following published results, each cited where it is used: results of Hibi and Li \cite[Theorem~2.1 and Lemma~1.1]{HL}, of Hibi, Li, Sahara and Shikama \cite[Theorem~2.1]{HLSS}, and of Stanley \cite[Corollary~1.3 and Theorem~2.2]{Stanley}; properties of vertex figures and face figures \cite[Proposition~2.4 and Chapter~2]{Ziegler}; and the standard facts about polytopes listed in Section~\ref{sec:defs}. Every other step is proved here, so Theorem~\ref{thm:main},
Corollary~\ref{cor:equality} and Theorem~\ref{thm:resolution}(b) depend on Fourier's preprint only through
Lemma~\ref{lem:identity} and Proposition~\ref{prop:transverse}.

An expanded version of this paper, with further results, is posted
on Zenodo \cite{G}. None of those results is used here.

\section{Preliminaries}\label{sec:defs}
\paragraph{Notation.} A \emph{chain} of a poset $P$ is a set of pairwise comparable elements, and an
\emph{antichain} is a set of pairwise incomparable ones; a chain is \emph{maximal} if no other chain
contains it. An \emph{order ideal} is a set $I\subseteq P$ with $p\in I$ whenever $p\leq p'\in I$, and
an \emph{upper set} is the complement of an order ideal. An element $r$ \emph{covers} $s$, and is an
\emph{upper cover} of $s$, if $s<r$ and no element lies strictly between them. A \emph{linear
extension} of $P$ is a listing $p_1,\ldots,p_n$ of its elements with $i<j$ whenever $p_i<p_j$. An \emph{induced subposet} of $P$ is a subset of $P$ with the restriction of the order of $P$. Throughout, inequalities between polynomials are understood coefficientwise.

\paragraph{Faces.} For a polytope $Q$, an affine inequality is \emph{valid} on $Q$ if every
point of $Q$ satisfies it; it is \emph{tight} at the points where equality holds, and it \emph{cuts out} the set of points of $Q$ at which it is tight. The \emph{faces} of $Q$ are the sets cut out in
this way; they include $Q$ and $\varnothing$. A \emph{facet} is a face of dimension $\dim Q-1$. The set of
points of $Q$ at which several valid inequalities are all tight is a face, since it is cut out by their sum; we say that these inequalities cut it out. If $Q$ is given as the set of solutions of a finite list of affine inequalities,
these are its \emph{defining inequalities}. For a nonempty face $F$ of $Q$, the \emph{defining inequalities} of $F$ are those of $Q$ together with the reverse of each defining inequality of $Q$ that is tight on $F$; by a fact listed below, $F$ is the set of their solutions. We use without proof the following standard facts about polytopes \cite[Chapters~1--2]{Ziegler}. A face of a face is a face, a face of $Q$ contained in a face $G$ of $Q$ is a face of $G$, and the intersection of two faces is a face. A nonempty
face $F$ of a polytope $Q$ is the set of points of $Q$ at which all defining inequalities of $Q$ that are
tight on $F$ are tight. A face other than $Q$ is \emph{proper}; a proper face has smaller dimension than $Q$ and lies in a facet. A polytope is
the convex hull of its vertices and the disjoint union of the relative interiors of its nonempty faces. The set of solutions of finitely many affine inequalities, if bounded, is a polytope, and a polytope has finitely many faces \cite[Theorem~1.1]{Ziegler}.
The vertices of the convex hull of a finite set $V$ lie in $V$. A \emph{polyhedron} is the set of solutions of finitely many affine inequalities and equations; a nonempty polyhedron, in particular a nonempty polytope, is the closure of its relative interior. An affine function that is nonnegative on a convex set and vanishes at a point of its relative interior vanishes on the whole set. Finally, no nonempty convex set is covered by finitely
many proper affine subspaces of its affine hull.

\paragraph{Order and chain polytopes.}
For a finite poset $P$, the \emph{order polytope}\footnote{Hibi and Li \cite{HL} use the inequalities $v_p\geq v_q$ for $p\leq q$; the unimodular map $v\mapsto\mathbf1-v$, where $\mathbf1$ is the all-ones vector, maps one version of $\OO(P)$ onto the other, so the results of Hibi and Li used in Example~\ref{ex:X} and Theorem~\ref{thm:HL} hold for either version.} and the \emph{chain polytope} are
\begin{gather*}
\OO(P)=\{v\in[0,1]^P: v_p\leq v_q \text{ whenever } p\leq q\},\\
\CC(P)=\Bigl\{x\in\mathbb R_{\geq0}^P:\sum_{p\in K}x_p\leq1
 \text{ for every chain } K\subseteq P\Bigr\}.
\end{gather*}
Both have dimension $|P|$ (see also Section~\ref{sec:cited}), and both are $0/1$ polytopes: the vertices of $\OO(P)$ are the indicator
vectors of the upper sets of $P$, and those of $\CC(P)$ are the indicator vectors of the antichains
of $P$ \cite[Corollary~1.3 and Theorem~2.2]{Stanley}. Sending an upper set to the set of its minimal elements is
a bijection from the upper sets onto the antichains, with inverse $A\mapsto\{r\in P:r\geq p\text{ for some }p\in A\}$,
so the two polytopes have the same number of vertices. Two $n$-dimensional polytopes $Q,Q'\subseteq\mathbb R^n$ with integer vertices are
\emph{unimodularly equivalent} (as in \cite{HL}) if $Q'=M(Q)+\mathbf w$
for an integral matrix $M$ with $\det M=\pm1$ and an integral vector
$\mathbf w$. The affine bijection $u\mapsto Mu+\mathbf w$ maps the $j$-faces of $Q$ onto those of $Q'$, so
unimodularly equivalent polytopes have the same $f$-vector.

\begin{example}\label{ex:X}
The facets of $\OO(P)$ correspond to the minimal elements, the maximal elements and the cover
relations of $P$, and those of $\CC(P)$ to the elements and the maximal chains of $P$
\cite[Lemma~1.1]{HL}. So $\OO(X)$ has $2+2+4=8$ facets and $\CC(X)$ has $5+4=9$; as unimodularly
equivalent polytopes have the same $f$-vector, $\OO(X)$ and $\CC(X)$ are not unimodularly equivalent.
Proposition~\ref{prop:X} shows that
$F_{\CC(X)}-F_{\OO(X)}=t^2+2t^3+t^4=t^2(1+t)^2$, the bound of
Theorem~\ref{thm:main} for $n=5$. Hence $\OO(X)$ and $\CC(X)$ have the same number of vertices and the same number of edges, and $\CC(X)$ has one more face of dimension $2$ than $\OO(X)$, two more of dimension $3$ and one more of dimension $4$, in agreement with the numbers of facets, $9$ for $\CC(X)$ and $8$ for $\OO(X)$.
\end{example}

\begin{theorem}[{\cite[Theorem~2.1]{HL}}]\label{thm:HL}
Let $P$ be a finite poset. Then $\OO(P)$ and $\CC(P)$ are unimodularly
equivalent if and only if $P$ does not contain $X$.
\end{theorem}

\paragraph{Pyramids.} For use in Section~\ref{sec:sharp}, we recall the faces of a pyramid and its face polynomial. A \emph{pyramid} $\pyr(Q)$ over a polytope $Q$
is the convex hull of $Q$ and a point $o$, the \emph{apex}, outside the affine hull $\aff(Q)$. Its nonempty faces are the nonempty faces $G$ of $Q$, the set $\{o\}$, and the pyramids $\operatorname{conv}(G\cup\{o\})$ over the nonempty faces $G$ of $Q$. As $o\notin\aff(G)$, the pyramid $\operatorname{conv}(G\cup\{o\})$ has dimension $\dim G+1$. These sets are pairwise distinct: no $G$ contains $o$, and $\operatorname{conv}(G\cup\{o\})\cap Q=G$ because $o\notin\aff(Q)$. To see that they are faces, note that an affine function $\psi$ on $\aff(Q)$ extends to an affine function $\tilde\psi$ on
$\aff(\pyr(Q))$ with any prescribed value at $o$. If $\psi\geq0$ on $Q$ and $G$ is the set of points of $Q$ at which $\psi$ vanishes,
then the inequality $\tilde\psi\geq0$, valid on $\pyr(Q)$ as it holds on $Q$ and at $o$, cuts out $G$ from $\pyr(Q)$ if $\tilde\psi(o)=1$, and $\operatorname{conv}(G\cup\{o\})$ if $\tilde\psi(o)=0$; for $\psi\equiv1$ and $\tilde\psi(o)=0$ it cuts out $\{o\}$. Conversely, let $F$ be a
nonempty face of $\pyr(Q)$. Its vertices are among $o$ and the vertices of $Q$. As $Q$ is a face of $\pyr(Q)$ (the case $\psi\equiv0$, $\tilde\psi(o)=1$ above), the set $G=F\cap Q$ is a face of $\pyr(Q)$ contained in $Q$, hence a face of $Q$; its vertices are the vertices of $F$ other than $o$. So
$F=G$ if $o\notin F$, $F=\{o\}$ if $G=\varnothing$, and $F=\operatorname{conv}(G\cup\{o\})$ otherwise. Hence
$F_{\pyr(Q)}(t)=1+(1+t)F_Q(t)$.

\section{Chain-order polytopes and Fourier's method}\label{sec:cited}
Fourier \cite{Fourier} passes from $\OO(P)$ to $\CC(P)$ through the chain-order polytopes of Fang and Fourier \cite{FF} and of Fang, Fourier, Litza and Pegel \cite{FFLP}. We follow Fourier's notation, except that we write $\OO_I(P)$ (defined below) for $\OO_{I,J}(P)$.

\paragraph{Results from Fourier's preprint.} Since that preprint has not yet been refereed, we state exactly what we
take from it. We adopt its Definition~4.2 (our Definition~\ref{def:transverse}), and we use without proof
its Lemma~4.1 (the first equality of our Lemma~\ref{lem:identity}) and its Proposition~5.2 (our
Proposition~\ref{prop:transverse}). We reprove its Proposition~3.1 (our Proposition~\ref{prop:fibres}),
Steps~1--2 of the proof of its Lemma~4.3 (our Lemma~\ref{lem:carrier}) and its Proposition~6.1 (our
Proposition~\ref{prop:mono}), as well as the elementary facts about $\OO_I(P)$ and about the subdivisions defined below that we need.

\paragraph{Chain-order polytopes.} Let $I\subseteq P$ be an order ideal and $J=P\setminus I$. The coordinates $x_p$ indexed by $p\in I$ are the \emph{chain coordinates}, and the coordinates $y_r$ indexed by $r\in J$ are the \emph{order coordinates}. The \emph{chain-order polytope}
$\OO_I(P)$ is defined by the
inequalities
\[
\begin{gathered}
x_p\geq0\ (p\in I),\qquad 0\leq y_r\leq1\ (r\in J),\qquad
y_r\leq y_{r'}\ (r<r' \text{ in } J),\\
\textstyle\sum_{p\in K}x_p\leq y_r\ \ (r\in J,\ K\subseteq I \text{ a
chain with } p<r \text{ for all } p\in K),\\
\textstyle\sum_{p\in K}x_p\leq1\ \ (K\subseteq I \text{ a chain}).
\end{gathered}
\]
We call the inequalities in the second and third lines of this display \emph{chain inequalities}, and the inequalities $y_r\leq y_{r'}$ \emph{order inequalities}. By definition, $\OO_\varnothing(P)=\OO(P)$ and $\OO_P(P)=\CC(P)$. Each $\OO_I(P)$ lies in $[0,1]^P$, since for $p\in I$ the chain $\{p\}$ gives $x_p\leq1$; so it is a polytope. It has dimension $|P|$: if
$r_1,\ldots,r_m$ is a linear extension of $J$, the point with $x_p=\frac1{4|P|}$ ($p\in I$) and
$y_{r_i}=\frac12+\frac{i}{4m}$ satisfies every inequality above strictly, as every sum $\sum_{p\in K}x_p$
over a chain $K\subseteq I$ is at most $\frac{|I|}{4|P|}\leq\frac14$.

\paragraph{From order coordinates to chain coordinates.}
Let $q$ be minimal in $J$, so that $I\cup\{q\}$ is again an order
ideal. We call the passage from $\OO_I(P)$ to $\OO_{I\cup\{q\}}(P)$ the
\emph{switch} of $q$. Let $n=|P|$, put $P'=P\setminus\{q\}$ with the induced order (then $I$ is an order ideal of $P'$),
and let $B=\OO_I(P')$, the \emph{base} of the switch, of dimension
$n-1$. For $u\in B$, with coordinates $x_p$ ($p\in I$) and $y_r$ ($r\in J\setminus\{q\}$), put
\begin{gather*}
L(u)=\max\Bigl\{\sum_{p\in K}x_p: K\subseteq I \text{ a chain with }
p<q \text{ for all } p\in K\Bigr\},\\
U(u)=\min\bigl(\{y_r:r>q\}\cup\{1\}\bigr).
\end{gather*}
The empty chain is allowed in the definition of $L$; its sum is $0$. Every $r>q$ lies in
$J\setminus\{q\}$, as $J$ is an upper set. We have $0\leq L\leq U\leq1$ on $B$: the empty chain gives $L\geq0$,
and for a chain $K$ as in the definition of $L$ and any $r>q$, the inequalities
$\sum_{p\in K}x_p\leq y_r$ and $\sum_{p\in K}x_p\leq1$ are defining inequalities of $B$, as $p<q<r$ for
$p\in K$. As the maximum of finitely many affine functions, $L$ is convex and piecewise affine; likewise, as the minimum of finitely many affine functions, $U$ is concave and piecewise affine.

\begin{proposition}[{\cite[Proposition~3.1]{Fourier}}]\label{prop:fibres}
In coordinates $(u,z)$, where $u\in B$ and $z$ is the coordinate indexed by $q$ ($z=y_q$
for $\OO_I(P)$ and $z=x_q$ for $\OO_{I\cup\{q\}}(P)$),
\begin{gather*}
\OO_I(P)=\{(u,z):u\in B,\ L(u)\leq z\leq U(u)\},\\
\OO_{I\cup\{q\}}(P)=\{(u,z):u\in B,\ 0\leq z\leq U(u)-L(u)\}.
\end{gather*}
\end{proposition}
\begin{proof}
As $q$ is minimal in $J$ and $I$ is an order ideal, every element comparable to $q$ is either in
$I$ and below $q$, or in $J$ and above $q$. Hence the defining inequalities of $\OO_I(P)$ that involve
$y_q$ are $0\leq y_q\leq1$, $y_q\leq y_r$ ($r>q$) and $\sum_{p\in K}x_p\leq y_q$ ($K\subseteq I$ a chain
below $q$), that is, $L(u)\leq y_q\leq U(u)$; the remaining inequalities are those of $B$. In
$\OO_{I\cup\{q\}}(P)$, the element $q$ is maximal in $I\cup\{q\}$, so the chains that contain $q$ are the
sets $K\cup\{q\}$ with $K$ as in the definition of $L$. The inequalities that involve $x_q$ are
$x_q\geq0$, $x_q+\sum_{p\in K}x_p\leq1$ and $x_q+\sum_{p\in K}x_p\leq y_r$ ($r>q$), that is,
$0\leq x_q\leq U(u)-L(u)$; again, the others are those of $B$.
\end{proof}
For example, let $P$ be the chain $p_1<p_2$. At the switch of $p_1$ (from $I=\varnothing$ to $I=\{p_1\}$) we
have $L\equiv0$, so the two polytopes of Proposition~\ref{prop:fibres} coincide: $\OO_{\{p_1\}}(P)=\OO(P)$, with
$x_{p_1}=v_{p_1}$. At the switch of $p_2$ (from $I=\{p_1\}$ to $I=P$) we have $B=[0,1]$, $L=x_{p_1}$ and
$U\equiv1$, and Proposition~\ref{prop:fibres} gives the triangles
$\OO_{\{p_1\}}(P)=\{(x_{p_1},z):0\leq x_{p_1}\leq z\leq1\}$ and
$\CC(P)=\{(x_{p_1},z):x_{p_1}\geq0,\ 0\leq z\leq1-x_{p_1}\}$, both with face polynomial $3+3t+t^2$.

\paragraph{Subdivisions.}
A \emph{subdivision} of a polytope $Q$ is a finite family of polytopes, called its \emph{cells}, such that the union of the cells is $Q$, every nonempty face of a cell is a cell, and any two cells meet in a common face (possibly empty). The relative interiors $\relint(C)$ of the cells $C$ then partition $Q$. Indeed, each point $x$ of a cell $C$ lies in the relative interior of the smallest face of $C$
that contains $x$, which is a cell; and if $x\in\relint(C)\cap\relint(C')$, then the face $C\cap C'$ of $C$
contains a relative interior point of $C$, so $C\cap C'=C$, and likewise $C\cap C'=C'$. A \emph{$k$-cell} is a
cell of dimension $k$, and a cell is \emph{maximal} if no other cell contains it. Call a set $D\subseteq B$ \emph{full-dimensional} if
$\dim D=\dim B=n-1$, and say that $L$ is \emph{affine on} $D$ if $L$ agrees on $D$ with an affine function. Let
$\SSS$ consist of the maximal (under inclusion) full-dimensional convex subsets of $B$ on which $L$ is affine,
together with all their nonempty faces, and
define $\TT$ in the same way for $U$; the next paragraph shows that these maximal sets are polytopes.
For example, for $P=X$ with $I=\{a,b\}$ and $q=c$, the maximal cells of $\SSS$ are $B\cap\{x_a\geq x_b\}$ and $B\cap\{x_a\leq x_b\}$,
and those of $\TT$ are $B\cap\{y_g\geq y_h\}$ and $B\cap\{y_g\leq y_h\}$ (Figure~\ref{fig:slice} and
Proposition~\ref{prop:X}).

\paragraph{The maximal cells of $\SSS$ and $\TT$.}
Write $L=\max_K\ell_K$, where $\ell_K(u)=\sum_{p\in K}x_p$ is the \emph{chain sum} of a chain $K$ in the definition of $L$. The maximal cells of
$\SSS$ are the full-dimensional ones among the sets $A_\ell=\{u\in B:L(u)=\ell(u)\}$, where $\ell$ runs over the chain sums $\ell_K$.
Indeed, $A_\ell=\{u\in B:\ell(u)\geq\ell_K(u)\text{ for all }K\}$ is a polytope, and $L$ equals the affine function $\ell$ on it. Now let $D\subseteq B$ be convex and full-dimensional, with $L$ affine on $D$. The finitely many convex sets $A_\ell\cap D$ cover $D$, so one of them, say $A_\ell\cap D$, is full-dimensional, since finitely many hyperplanes cannot cover a full-dimensional convex set. The affine functions $L|_D$ and $\ell$ agree on that set, hence on $D$; so $D\subseteq A_\ell$. If $A_\ell\subseteq A_{\ell'}$ and $A_\ell$ is full-dimensional, then $\ell=\ell'$, as $\ell=L=\ell'$ on $A_\ell$. Hence every maximal
such set $D$ is a full-dimensional $A_\ell$; conversely, if a full-dimensional $A_\ell$ lies in a convex
set $D\subseteq B$ on which $L$ is affine, then $D\subseteq A_{\ell'}$ for some $\ell'$, so $\ell=\ell'$ and $D=A_\ell$. So the full-dimensional sets $A_\ell$ are the maximal sets in the definition of $\SSS$; as every cell of $\SSS$ is a face of one of them and none of them contains another, they are its maximal cells. The full-dimensional sets $A_\ell$ cover $B$. Indeed, every point of $B$ lies in some $A_\ell$, so the union of the full-dimensional sets $A_\ell$ contains every point of $B$ outside the lower-dimensional sets $A_\ell$. These finitely many sets lie in finitely many hyperplanes of $\aff(B)$, which cover no nonempty open subset of $B$ (Section~\ref{sec:defs}). So the points of $B$ outside them are dense in $B$, and, as the union is closed, it equals $B$. The same holds for $\TT$, with $U$ and its terms $y_r$ and $1$ in place of $L$ and the chain sums.

\paragraph{The common refinement.}
Let $\RR=\SSS\vee\TT$ be the \emph{common refinement} of $\SSS$ and $\TT$, whose
cells are the nonempty intersections $S\cap T$ with $S\in\SSS$,
$T\in\TT$. We now check that all three are subdivisions of $B$. The maximal cells of $\SSS$ and of $\TT$ cover $B$, as just
shown, so the cells $S\cap T$ do too. Next, $A_\ell\cap A_{\ell'}=A_\ell\cap\{\ell=\ell'\}$ is a face of
$A_\ell$ (as $\ell\geq\ell'$ on $A_\ell$) and of $A_{\ell'}$. If $F$ and $F'$ are faces of $A_\ell$ and of $A_{\ell'}$, respectively, then
$F\cap A_{\ell'}$ and $F'\cap A_\ell$ are faces of the polytope $A_\ell\cap A_{\ell'}$. Hence their intersection $F\cap F'$ is a face of $A_\ell\cap A_{\ell'}$, so of $A_\ell$ and of $A_{\ell'}$, and therefore of $F$ and of $F'$. Thus any two cells of $\SSS$ meet in a common face; the same holds for $\TT$. Each cell of $\SSS$ is a face of some $A_\ell$, whose defining inequalities are those of $B$ and the inequalities $\ell\geq\ell_K$ for all $K$; so the cells of $\SSS$, and likewise those of $\TT$, have defining inequalities in the sense of Section~\ref{sec:defs}. As $S\cap T$ is the set of solutions of the defining inequalities of $S$ and of $T$ together, a nonempty face of $S\cap T$ is $S_0\cap T_0$, where $S_0$ and $T_0$ are the faces of $S$ and $T$ cut out by the defining inequalities of $S$ and of $T$ that are tight on that face; so $\RR$ contains the nonempty faces of its cells. Finally,
$(S\cap T)\cap(S'\cap T')=(S\cap S')\cap(T\cap T')$ is cut out of $S\cap T$ by valid inequalities of
$S$ and of $T$, so it is a face of $S\cap T$, and likewise of $S'\cap T'$.

\paragraph{Restrictions and carriers.}
For a face $E$ of $B$, the \emph{restriction} $\SSS|_E$ consists of the
nonempty sets $S\cap E$ with $S\in\SSS$; each is a face of $S$, because a valid inequality of $B$ that cuts out $E$ is valid on $S$. So $\SSS|_E$ is a subdivision of $E$ into cells of $\SSS$: its cells cover $E$, as those of $\SSS$ cover $B$; a nonempty face of a cell $S\cap E$ is a cell of $\SSS$ contained in $E$; and two cells of $\SSS|_E$ meet in a common face, as cells of $\SSS$ do. Similarly for $\TT|_E$
and $\RR|_E$. The \emph{carrier} of a cell is the smallest face of $B$ that contains it; it exists, as the intersection of all faces of $B$ that contain the cell is such a face. The relative interior of a cell $C$ lies in the relative interior of its carrier. Indeed, if a point of $\relint(C)$ lay in a proper face $G$ of the carrier, an inequality valid on the carrier that cuts out $G$ would be tight at a relative interior point of $C$, hence on all of $C$ (Section~\ref{sec:defs}); so $C\subseteq G$, contrary to the minimality of the carrier. As the
relative interiors of the faces of $B$ partition $B$, the carrier of $C$ is the only face of $B$ whose
relative interior meets $\relint(C)$.

\paragraph{Local terms.}
For a subdivision $\mathcal Y$, let $F_{\mathcal Y}(t)$ be the polynomial
whose coefficient of $t^k$ is the number of $k$-cells of $\mathcal Y$.
For a nonempty $e$-face $E$ of $B$ and $k\geq0$, let $\sigma_E(k)$, $\tau_E(k)$ and $\rho_E(k)$ be the numbers of $k$-cells with carrier $E$ in $\SSS$, $\TT$ and $\RR$, respectively, and put
\[
\delta_{E,k}=\rho_E(k)-\sigma_E(k)-\tau_E(k)+[k=e],\qquad
\delta_E(t)=\sum_k\delta_{E,k}\,t^k ,
\]
where $[k=e]$ is $1$ if $k=e$ and $0$ otherwise. We call $\delta_E(t)$
the \emph{local term} of $E$. The subdivisions $\SSS$, $\TT$ and $\RR$, the polynomials $F_{\mathcal Y}$ and the carriers are those of Fourier \cite[Section~4]{Fourier}. In Fourier's preprint, $\SSS$ and $\TT$ are ``the subdivisions of $B$ into maximal domains on which $L$ and $U$ are affine, together with their faces''; the maximal cells of $\SSS$ are the sets $\{u\in B:L(u)=\ell(u)\}$ with nonempty interior relative to $\aff(B)$, that is, our full-dimensional sets $A_\ell$; and $\RR$ is their common refinement. So Lemma~\ref{lem:identity} and Proposition~\ref{prop:transverse} below apply to them as stated.

\begin{lemma}[{\cite[Lemma~4.1]{Fourier}, first equality}]\label{lem:identity}
For every switch, $F_{\OO_{I\cup\{q\}}(P)}-F_{\OO_I(P)}=F_B+F_\RR-F_\SSS-F_\TT
=\sum_{E}\delta_E(t)$, the sum over all nonempty faces $E$ of $B$.
\end{lemma}
The second equality holds because every cell has a unique carrier and $F_B=\sum_E t^{\dim E}$. In the example
after Proposition~\ref{prop:fibres}, $L$ and $U$ are affine, so $\SSS$, $\TT$ and $\RR$ all consist of the nonempty faces of $B$, and Lemma~\ref{lem:identity} shows that the difference is $0$. 

\paragraph{An informal reading (not used below).} A face of $\OO_I(P)$ that contains no segment parallel to
the $z$-axis is the graph of $L$ over a cell of $\SSS$, or the graph of $U$ over a cell of $\TT$, or
both. After the switch, a face with no such segment is the part of the facet $\{z=0\}$ over a face of $B$, or the graph
of $U-L$ over a cell of $\RR$, or both. In both polytopes, the faces that are of both kinds are those over the faces of $B$ on which $L=U$, so the two polytopes have the same number of them. Thus $F_B+F_\RR-F_\SSS-F_\TT$ counts, dimension by dimension, the change in the number of faces with no segment parallel to the $z$-axis, and the first equality of Lemma~\ref{lem:identity}
is equivalent to the statement that the two polytopes have, in each dimension, the same number of faces that contain such a segment.

\begin{definition}[{\cite[Definition~4.2]{Fourier}}]\label{def:transverse}
The subdivisions $\SSS$ and $\TT$ are \emph{transverse relative to the faces of $B$} if,
for every face $E$ of $B$ and every $u\in\relint(E)$, the unique cells
$S$ of $\SSS|_E$ and $T$ of $\TT|_E$ whose relative interiors contain
$u$ satisfy $\dim(S\cap T)=\dim S+\dim T-\dim E$.
\end{definition}

Although $L$ depends only on the coordinates $x_p$ with $p<q$ and $U$ only
on the coordinates $y_r$ with $r>q$, these coordinates are coupled by the inequalities of $B$,
so the transversality of $\SSS$ and $\TT$ relative to the faces of $B$ needs a proof. The idea of Fourier's proof is this. The walls of $\SSS$ (Section~\ref{sec:intro}) are cut out by affine forms in the coordinates $x_p$ with $p<q$, and those of $\TT$ by forms in the coordinates $y_r$ with $r>q$. Let $E$ be a face of $B$ and $u\in\relint(E)$. If $L(u)=U(u)$, then $L=U$ on $E$, as $U-L$ is concave and nonnegative; so on $E$ the convex function $L$ equals the concave function $U$, hence is affine, and $\SSS|_E$ and $\TT|_E$ consist of the nonempty faces of $E$. Otherwise choose $\theta$ with $L(u)<\theta<U(u)$. Each point $u'$ of $E$ near $u$ splits as $u'=u^-+u^+$ with $u^-\in\theta E$ and $u^+\in(1-\theta)E$, where $u^-$ has the same coordinates $x_p$ ($p<q$) as $u'$ and has $y_r=\theta$ for every $r>q$ \cite[Lemma~5.1 and Step~1 of the proof of Proposition~5.2]{Fourier}. From this splitting Fourier deduces the following. Every vector $d$ parallel to $\aff(E)$ splits in the same way into a sum $d=d^-+d^+$, where the linear parts of the forms of $\TT$ vanish on $d^-$ and those of $\SSS$ vanish on $d^+$. So, restricted to the linear space parallel to $\aff(E)$, the linear parts of the two families of forms span subspaces that meet only in $0$, and the cells of $\SSS$ and $\TT$ through $u$ meet transversely.

\begin{proposition}[{\cite[Proposition~5.2]{Fourier}}]\label{prop:transverse}
The subdivisions $\SSS$ and $\TT$ are transverse relative to the faces of $B$.
\end{proposition}

The next lemma records, for each carrier $E$ separately, facts that Fourier \cite[Lemma~4.3]{Fourier} states only after summing over all carriers;
we include the proof.

\begin{lemma}[{\cite[proof of Lemma~4.3, Steps~1--2]{Fourier}}]\label{lem:carrier}
Assume that $\SSS$ and $\TT$ are transverse relative to the faces of $B$.
Let $E$ be an $e$-face of $B$. For a cell $Y$ of $\SSS$ or of $\TT$, let $\RR(Y)$ be the set of
$(\dim Y)$-cells of $\RR$ contained in $Y$.
\begin{enumerate}
\item[(i)] Let $k<e$. For every $k$-cell $Y$ of $\SSS$ or of $\TT$ with
carrier $E$, the set $\RR(Y)$ is nonempty and its members have carrier
$E$. The sets $\RR(Y)$ for all such $Y$ in $\SSS$ and in $\TT$ are pairwise disjoint.
\item[(ii)] Let $G_E$ be the bipartite graph whose vertex set is the
disjoint union of the sets of $e$-cells of $\SSS|_E$ and of $\TT|_E$, with an edge $\{S,T\}$ for
every $e$-cell $S\cap T$ of $\RR|_E$ (each $e$-cell of $\RR|_E$ has this form for exactly one such
pair: the $e$-cells of $\SSS|_E$ and $\TT|_E$ that contain it). Then $G_E$ is connected.
\item[(iii)] We have $\delta_{E,k}\geq0$ for every $k$.
\end{enumerate}
\end{lemma}
\begin{proof}
(i) By symmetry we may take $Y\in\SSS$. The cells $Y\cap T$ ($T\in\TT$) of $\RR$ cover $Y$, and finitely many polytopes of
dimension less than $\dim Y$ cannot cover $Y$; so $\RR(Y)\neq\varnothing$. A member $R$ of $\RR(Y)$ satisfies
$R\subseteq Y$ and $\dim R=\dim Y$, so $\relint(R)\subseteq\relint(Y)\subseteq\relint(E)$, and $R$ has
carrier $E$. Distinct cells of
the same subdivision have disjoint relative interiors, so the sets $\RR(Y)$ with $Y\in\SSS$ are pairwise disjoint, and so are those with $Y\in\TT$. If $R\in\RR(S)\cap\RR(T)$ with
$S\in\SSS$ and $T\in\TT$, then a point of $\relint(R)$ lies in
$\relint(S)\cap\relint(T)\cap\relint(E)$, and transversality gives
$\dim(S\cap T)=2k-e<k=\dim R$, impossible as $R\subseteq S\cap T$. In particular (take $S=T=Y$), no $k$-cell
$Y$ with carrier $E$ is a cell of both $\SSS$ and $\TT$.

(ii) An $e$-cell $R$ of $\RR|_E$ is $S\cap T$ with $S\in\SSS|_E$ and $T\in\TT|_E$, and $\dim R=e$ forces
$\dim S=\dim T=e$; an $e$-cell of $\SSS|_E$ that contains $R$ contains $\relint(R)$ in its relative interior,
so it is unique, and likewise for $\TT|_E$. The \emph{dual graph} of the subdivision $\RR|_E$ of $E$ has the $e$-cells as vertices, two of them being adjacent when they share an $(e-1)$-face; it is connected \cite[Lemma~2]{CJK}. We recall the argument. Let $C_1$ and $C_2$ be distinct $e$-cells. Fix $x\in\relint(C_1)$, and choose $y\in\relint(C_2)$ outside the finitely many proper affine subspaces $\aff(\{x\}\cup C)$ of
$\aff(E)$, where $C$ runs over all cells of dimension at most $e-2$ and over those $(e-1)$-cells with $x\in\aff(C)$; this is possible, as $\relint(C_2)$ is a nonempty open subset of $\aff(E)$. The segment $[x,y]$ meets no cell $C$ of these two kinds. Otherwise it meets such a cell $C$ at a point $x^*$, and $x^*\neq x$, since $x$ lies in no cell of dimension less than $e$ (the relative interiors of the cells partition $E$). Then the line through $x$ and $x^*$, which contains $y$, lies in $\aff(\{x\}\cup C)$, a contradiction. Every other $(e-1)$-cell $C$ satisfies $x\notin\aff(C)$, so the segment $[x,y]$ meets the hyperplane $\aff(C)$ of $\aff(E)$, and hence $C$, in at most one point. Hence all but finitely many points of $[x,y]$ lie in relative
interiors of $e$-cells; we call the other points of $[x,y]$ \emph{exceptional}. The relative interiors of the
$e$-cells are disjoint and open in $\aff(E)$, so each open subsegment of $[x,y]$ between two consecutive points of the finite set formed by $x$, $y$ and the exceptional points lies in the relative interior of one $e$-cell. An exceptional point $x'$
lies in the relative interior of an $(e-1)$-cell $D$. The $e$-cells $C_-$ and $C_+$ that contain the subsegments next to $x'$ contain $x'$, so $C_-\cap D$ is a face of $D$ that contains the relative interior point $x'$ of $D$; hence $D\subseteq C_-$, and likewise $D\subseteq C_+$. So $D$ is a common $(e-1)$-face of $C_-$ and $C_+$,
and consecutive $e$-cells along $[x,y]$ are adjacent. The first and the last subsegment lie in $\relint(C_1)$ and in $\relint(C_2)$, respectively, as these sets are open in $\aff(E)$ and contain $x$ and $y$; so $C_1$ and $C_2$ are joined by a path in the dual graph.

Let $S\cap T$ and $S'\cap T'$ be adjacent $e$-cells of
$\RR|_E$, with common $(e-1)$-face $D$. The two cells have disjoint relative interiors and both
have $D$ as a facet. They lie on opposite sides of the hyperplane $\aff(D)$ of $\aff(E)$. Indeed, each lies on one
side, as $D$ is a facet of it. Fix $x_D\in\relint(D)$. Near $x_D$, each cell coincides with the closed half-space bounded by $\aff(D)$ that contains it. A defining inequality of the cell that is tight at $x_D$ is tight on $D$; within $\aff(E)$, it is therefore either tight on the whole cell or its set of solutions is that half-space, and the other defining inequalities remain strict near $x_D$. Thus, if both cells lay on the same side, they would share relative interior points. Hence $D$ lies in no facet of $E$, since otherwise $E$ would lie on one side of $\aff(D)$, yet it contains $e$-cells on both sides. So $D$ lies in no proper face of $E$, and $E$ is the carrier of $D$. Hence $\relint(D)\subseteq\relint(E)$, and transversality applies at the points of $\relint(D)$.
Suppose $S\neq S'$ and $T\neq T'$. Then $S\cap S'$ is a proper face of the $e$-cell $S$, and it contains $D$; so it is an $(e-1)$-cell; likewise $T\cap T'$. As $D$ has dimension $e-1$, a point
of $\relint(D)$ lies in $\relint(S\cap S')$ and in $\relint(T\cap T')$,
and transversality gives $\dim(S\cap S'\cap T\cap T')=e-2<\dim D$, a
contradiction. Thus adjacent $e$-cells of $\RR|_E$ are edges of $G_E$
with a common vertex, and, as the dual graph is connected, all edges of
$G_E$ lie in the same component. Every vertex lies on an edge, since every
$e$-cell of $\SSS|_E$ or $\TT|_E$ contains an $e$-cell of $\RR|_E$ (as in the proof of (i)).
Hence $G_E$ is connected.

(iii) For $k>e$ there are no cells with carrier $E$. For $k<e$, part (i) gives $\rho_E(k)\geq\sigma_E(k)+\tau_E(k)$. For $k=e$, the cells with
carrier $E$ in $\SSS$ and $\TT$ are the $e$-cells of $\SSS|_E$ and $\TT|_E$ (an $e$-cell in $E$ has
carrier $E$), that is, the vertices of $G_E$, and those in
$\RR$ are its edges; a connected graph with $\nu$ vertices has at least $\nu-1$ edges,
so $\delta_{E,e}\geq0$.
\end{proof}

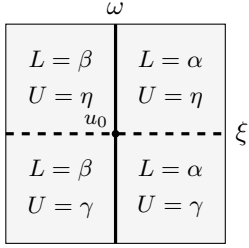
\begin{figure}[!tb]
\centering
\begin{minipage}{0.28\textwidth}
\centering
\begin{tikzpicture}[scale=1.45]
 \fill[gray!8] (-1,-1) rectangle (1,1);
 \draw (-1,-1) rectangle (1,1);
 \draw[very thick] (0,-1) -- (0,1);
 \draw[very thick,dashed] (-1,0) -- (1,0);
 \fill (0,0) circle (0.035);
 \node[above left,inner sep=2.5pt] at (0,0) {\scriptsize $u_0$};
 \node[align=center] at (0.5,0.5) {\footnotesize $L=\alpha$\\ \footnotesize $U=\eta$};
 \node[align=center] at (-0.5,0.5) {\footnotesize $L=\beta$\\ \footnotesize $U=\eta$};
 \node[align=center] at (0.5,-0.5) {\footnotesize $L=\alpha$\\ \footnotesize $U=\gamma$};
 \node[align=center] at (-0.5,-0.5) {\footnotesize $L=\beta$\\ \footnotesize $U=\gamma$};
 \node[right] at (1,0) {\small $\xi$};
 \node[above] at (0,1) {\small $\omega$};
\end{tikzpicture}
\end{minipage}\hfill
\begin{minipage}{0.70\textwidth}
\caption{The subdivisions $\SSS$ and $\TT$ for $P=X$ at the switch of $c$
(Proposition~\ref{prop:X}). Here $B=E_0$ (Section~\ref{subsec:E0}), with coordinates $\alpha=x_a$, $\beta=x_b$, $\gamma=y_g$, $\eta=y_h$; on $B$ we have $L=\max(\alpha,\beta)$ and $U=\min(\gamma,\eta)$. The figure shows the slice of $B$ by the plane
$\{\alpha+\beta=\frac12,\ \gamma+\eta=\frac32\}$ through
$u_0=(\frac14,\frac14,\frac34,\frac34)$, in the coordinates
$\xi=\alpha-\beta$, $\omega=\gamma-\eta$ ($|\xi|,|\omega|\leq\frac12$).
The function $L$ is not affine on any neighborhood of a point with $\xi=0$; these points lie in the $3$-cell $B\cap\{\alpha=\beta\}$ of $\SSS$ (solid line). Likewise, $U$ is not affine on any neighborhood of a point with $\omega=0$; these points lie in the $3$-cell $B\cap\{\gamma=\eta\}$ of $\TT$ (dashed line). The point
$u_0$, the four half-axes and the four quadrants are the slices of the $2$-, $3$- and $4$-dimensional cells of $\RR$ with carrier $B$, respectively, while $\SSS$ and $\TT$ each have two
$4$-cells and a $3$-cell with carrier $B$; so
$\delta_B=(1-0-0)t^2+(4-1-1)t^3+(4-2-2+1)t^4=t^2(1+t)^2$.}
\label{fig:slice}
\end{minipage}
\end{figure}

\begin{lemma}\label{lem:restrict}
Let $E$ be a face of $B$ of dimension $e$. The $e$-cells of $\SSS|_E$
are exactly the maximal $e$-dimensional convex subsets of $E$ on which
$L$ is affine, and every $e$-dimensional convex subset of $E$ on which
$L$ is affine lies in one of them. The same holds for $\TT|_E$ and $U$.
\end{lemma}
\begin{proof}
An $e$-cell of $\SSS|_E$ is a set $S\cap E$ with $S$ a face of a maximal
cell $A_\ell$ of $\SSS$. A valid inequality of $B$ that cuts out $E$ is valid on
$A_\ell\subseteq B$, so $A_\ell\cap E$ is a face of $A_\ell$, and so is
$S\cap E=S\cap(A_\ell\cap E)$. Since $S\cap E$ is an $e$-face of $A_\ell\cap E$,
which has dimension at most $e$, the two are equal: $S\cap E=A_\ell\cap E=\{u\in
E:L(u)=\ell(u)\}$, on which $L$ is affine. Now let $D\subseteq E$ be convex
of dimension $e$ with $L$ affine on $D$. The finitely many sets
$A_\ell\cap E$ with $A_\ell$ a maximal cell of $\SSS$ cover $E$, and their intersections with $D$ are
convex; if all had dimension less than $e$, they would lie in finitely many proper affine subspaces of
$\aff(D)$, which cannot cover $D$. Hence one of them, say $A_\ell\cap E\cap D$, contains a nonempty relatively open subset of $D$;
so $L=\ell$ on $D$,
that is, $D$ lies in $A_\ell\cap E$. This set is a face of $A_\ell$, as shown above, and it has dimension $e$, as it contains $D$; hence it is an $e$-cell of $\SSS|_E$. If two $e$-cells satisfy $A_\ell\cap E\subseteq A_{\ell'}\cap E$, then $\ell=\ell'$ on $\aff(E)$, so the two cells are equal. Hence each $e$-cell of $\SSS|_E$ is a maximal $e$-dimensional convex
subset of $E$ on which $L$ is affine, and each such maximal subset is the $e$-cell that contains it.
For $U$, use the minimum in place of the maximum.
\end{proof}

\begin{proposition}[{\cite[Proposition~6.1]{Fourier}}]\label{prop:mono}
For every switch, $F_{\OO_I(P)}\leq F_{\OO_{I\cup\{q\}}(P)}$.
\end{proposition}
\begin{proof}
By Lemma~\ref{lem:identity}, the difference is $\sum_E\delta_E(t)$. By
Proposition~\ref{prop:transverse}, Lemma~\ref{lem:carrier} applies, and part (iii) gives
$\delta_E(t)\geq0$ for every face $E$ of $B$.
\end{proof}

\section{The lower bound}\label{sec:gap}
\begin{theorem}\label{thm:main}
If an $n$-element poset $P$ contains $X$, then
$F_{\CC(P)}-F_{\OO(P)}\geq t^2(1+t)^{n-3}$ coefficientwise. For each
$n\geq5$, some $n$-element poset containing $X$ attains equality
(Proposition~\ref{prop:sharp}).
\end{theorem}

Lemma~\ref{lem:crossing} bounds the local term of a face of $B$ in which a wall of
$\SSS$ crosses a wall of $\TT$. Section~\ref{subsec:E0} chooses a switch and a $4$-face
$E_0$ of its base such that every face $E\supseteq E_0$ contains such a crossing
(Lemma~\ref{lem:coface-crossing}). Lemma~\ref{lem:coface-count} counts these faces, and
Section~\ref{subsec:sum} adds up the bounds.

\subsection{A local lower bound}
\begin{lemma}\label{lem:crossing}
Let $\SSS$ and $\TT$ be the subdivisions of the base $B$ of a switch, let $E$
be an $e$-face of $B$, and let $w\in\relint(E)$. Assume that some convex set $N\subseteq E$, open in $\aff(E)$ and
containing $w$, meets exactly two $e$-cells $S_1$ and $S_2$ of $\SSS|_E$, and that
$S_1\cap N$ and $S_2\cap N$ are the intersections of $N$ with the two
closed half-spaces of $\aff(E)$ bounded by a hyperplane $H_\SSS$ of $\aff(E)$ through
$w$. Assume the same for two $e$-cells $T_1$ and $T_2$ of $\TT|_E$ and a
hyperplane $H_\TT$ of $\aff(E)$ through $w$, and that
$H_\SSS\neq H_\TT$ (this forces $e\geq2$). Then
$\delta_E(t)\geq t^{e-2}(1+t)^2$.
\end{lemma}
Figure~\ref{fig:slice} shows this situation for $P=X$.

\begin{proof}
By Proposition~\ref{prop:transverse}, Lemma~\ref{lem:carrier} applies.
The common face $S_w=S_1\cap S_2$ meets $N$ in $H_\SSS\cap N$, a
neighborhood of $w$ in $H_\SSS$. As $S_1\neq S_2$, the face $S_w$ of $S_1$ is proper: otherwise $S_1$ would be an $e$-dimensional face of the $e$-cell $S_2$, and so $S_1=S_2$. Hence $\dim S_w\leq e-1$; as $S_w$ contains the relatively open subset $H_\SSS\cap N$ of $H_\SSS$, we get
$\dim S_w=e-1$, $\aff(S_w)=H_\SSS$ and $w\in\relint(S_w)$. So $S_w$ is the $(e-1)$-cell of
$\SSS|_E$ whose relative interior contains $w$. Likewise for
$T_w=T_1\cap T_2$ and $H_\TT$. Both have carrier $E$, as
$w\in\relint(E)$.

\emph{The coefficient of $t^{e-2}$.} The cell $R_w=S_w\cap T_w$ of $\RR$
meets $N$ in $H_\SSS\cap H_\TT\cap N$. As $R_w\subseteq H_\SSS\cap H_\TT$, an affine space of dimension
$e-2$ in which $H_\SSS\cap H_\TT\cap N$ is open, $R_w$ has dimension $e-2$ and
$w\in\relint(R_w)$; hence its carrier is $E$. Suppose $R_w\in\RR(Y)$ for
an $(e-2)$-cell $Y$ of $\SSS$ or of $\TT$ with carrier $E$. Then
$w\in\relint(R_w)\subseteq\relint(Y)$, so $Y=S_w$ or $Y=T_w$, whose dimension is
$e-1$, a contradiction. Thus $R_w$ lies in none of the
sets $\RR(Y)$ of Lemma~\ref{lem:carrier}(i), and
$\rho_E(e-2)\geq\sigma_E(e-2)+\tau_E(e-2)+1$.

\emph{The coefficient of $t^{e-1}$.} Since $H_\SSS\neq H_\TT$ and both
contain $w$, the set $H_\SSS\cap N$ meets both open sides of $H_\TT$ in $N$,
that is, $(T_1\setminus T_2)\cap N$ and $(T_2\setminus T_1)\cap N$. So $S_w\cap T_1$
and $S_w\cap T_2$ are distinct cells of $\RR$ in $S_w$, each containing
a relatively open subset of $H_\SSS$, hence of dimension $e-1$. In the same way, $T_w\cap S_1$ and
$T_w\cap S_2$ are distinct $(e-1)$-cells of $\RR$ in $T_w$. Thus
$|\RR(S_w)|\geq2$ and $|\RR(T_w)|\geq2$. By Lemma~\ref{lem:carrier}(i), the sets $\RR(Y)$, where $Y$ runs over the $(e-1)$-cells of $\SSS$ and of $\TT$ with carrier $E$, are nonempty and pairwise disjoint, and their members are $(e-1)$-cells of $\RR$ with carrier $E$. Hence $\rho_E(e-1)\geq\sigma_E(e-1)+\tau_E(e-1)+2$.

\emph{The coefficient of $t^{e}$.} As $H_\SSS\neq H_\TT$, each of the four
sets $S_i\cap T_j$ contains points of $N$ that lie on neither hyperplane. Each
$S_i\cap T_j$ is a cell of $\RR$ by the definition of the common
refinement; it has dimension $e$ since it contains a nonempty open subset of $N$, so
it is an edge of $G_E$, and $G_E$ contains the cycle $S_1T_1S_2T_2$. As
$G_E$ is connected (Lemma~\ref{lem:carrier}(ii)) and has a cycle, it has
at least as many edges as vertices: $\rho_E(e)\geq\sigma_E(e)+\tau_E(e)$,
that is, $\delta_{E,e}\geq1$. The other coefficients are nonnegative by
Lemma~\ref{lem:carrier}(iii). Adding these bounds gives
$\delta_E(t)\geq t^{e-2}+2t^{e-1}+t^e=t^{e-2}(1+t)^2$.
\end{proof}

\subsection{Construction of the face \texorpdfstring{$E_0$}{E0}}\label{subsec:E0}
Let $P$ contain $X$, let $a,b,c,g,h$ be five elements of $P$ with $a,b<c<g,h$ as in Section~\ref{sec:intro}, and write ${\uparrow}p=\{r\in P:r\geq p\}$ for
$p\in P$. Choose $q$ maximal among the elements $s\geq c$ of $P$ whose \emph{strict upper set} $\{r\in P:r>s\}$ is not a chain. The element $c$ is one of them, since $g$ and $h$ lie above $c$ and are incomparable. We cannot always take $q=c$: the construction of $E_0$ below needs
two incomparable upper covers of $q$, and $c$ may have only one, for example when an element $c'$ with $c<c'<g,h$ is the only upper cover of $c$.

\begin{lemma}\label{lem:q}
The element $q$ has at least two upper covers, and any two of them are
incomparable. The set $I=P\setminus{\uparrow}q$
is an order ideal, $q$ is minimal in $J={\uparrow}q$, and $a,b\in I$
with $a,b<q$.
\end{lemma}
\begin{proof}
The strict upper set of $q$ is not a chain, hence nonempty. As $P$ is finite, for every $r>q$ the set $\{s':q<s'\leq r\}$ has a minimal element $s$; it is an upper cover of $q$, and $r\geq s$. So $q$ has an upper
cover, and if $q$ had only
one upper cover $q^+$, every element $r>q$ would satisfy
$r\geq q^+$. Then any two incomparable elements above $q$ would lie strictly above $q^+$. Indeed, neither of them is $q^+$, as $q^+$ is comparable to every element above $q$. So the strict upper set of $q^+$ would not be a chain, and $q^+$, which satisfies $q^+>q\geq c$, would contradict the maximality of $q$. Two distinct upper covers are
incomparable: if $g'<h'$ for upper covers $g',h'$ of $q$, then $q<g'<h'$, so $h'$
does not cover $q$. As ${\uparrow}q$ is an upper set, its complement $I$ is an order ideal, and $q$, the minimum of ${\uparrow}q$, is minimal in $J$. Finally, since $a,b<c\leq q$, neither $a$ nor $b$ lies in ${\uparrow}q$, that is, $a,b\in I$.
\end{proof}

Fix two distinct upper covers $g'$ and $h'$ of $q$; they need not be $g$ and $h$. Consider the switch of $q$
with this $I$, its base $B=\OO_I(P\setminus\{q\})$, of dimension $n-1$,
and its functions $L$ and $U$. The coordinates of $B$ are $x_p$ for
$p\in I$ and $y_r$ for $r\in J'=J\setminus\{q\}$; every $r\in J'$
satisfies $r>q$. Put $\alpha=x_a$, $\beta=x_b$, $\gamma=y_{g'}$,
$\eta=y_{h'}$, and let $\zeta$ denote the vector of all other
coordinates of $B$; it ranges over the \emph{$\zeta$-space} $\mathbb R^Z$, where
$Z=(I\setminus\{a,b\})\cup(J'\setminus\{g',h'\})$. Let $\zeta_0$ be the point of $\mathbb R^Z$ with $x_p=0$ for
$p\in I\setminus\{a,b\}$ and $y_r=1$ for $r\in J'\setminus\{g',h'\}$,
and let $E_0$ be the set of points of $B$ with $\zeta=\zeta_0$.
On $E_0$ all other chain coordinates vanish and all other order coordinates equal $1$,
so $L=\max(\alpha,\beta)$ and $U=\min(\gamma,\eta)$ there, and $E_0$ is a copy of the base
of the switch of $c$ in $X$ (Figure~\ref{fig:slice}). Lemmas~\ref{lem:product}
and~\ref{lem:coface-crossing} carry the crossing of the walls $E_0\cap\{\alpha=\beta\}$ and $E_0\cap\{\gamma=\eta\}$ to every face of $B$ that contains $E_0$.

\begin{lemma}\label{lem:E0}
We have $E_0=\{(\alpha,\beta,\gamma,\eta):\alpha,\beta\geq0,\
\max(\alpha,\beta)\leq\min(\gamma,\eta),\ \gamma,\eta\leq1\}
\times\{\zeta_0\}$, a $4$-face of $B$.
\end{lemma}
\begin{proof}
The inequality $\sum_{p\in I\setminus\{a,b\}}x_p+\sum_{r\in J'\setminus\{g',h'\}}(1-y_r)\geq0$ is valid on $B$
and is tight exactly where $\zeta=\zeta_0$, so $E_0$ is a face of $B$. On $E_0$ every $x_p$ with $p\notin\{a,b\}$ vanishes, and a chain contains at most one of
$a$ and $b$; so each chain inequality of $B$ reads $\alpha\leq y_r$, $\beta\leq y_r$ or $0\leq y_r$ for some
$r\in J'$, or $\alpha\leq1$, $\beta\leq1$ or $0\leq1$. Since $a,b<q<r$ for every $r\in J'$, the chains $\{a\}$
and $\{b\}$ give all the inequalities $\alpha\leq y_r$ and $\beta\leq y_r$. As $y_r\in\{\gamma,\eta,1\}$ on
$E_0$, the chain inequalities other than $0\leq y_r$ and $0\leq1$ amount to $\max(\alpha,\beta)\leq\min(\gamma,\eta,1)$. An order inequality $y_r\leq y_{r'}$ ($r<r'$
in $J'$) becomes $1\leq1$, $\gamma\leq1$ or $\eta\leq1$ unless
$r'\in\{g',h'\}$, which is impossible, as $q<r<r'$ and both $g'$ and $h'$ cover $q$.
The inequalities $x_p\geq0$ and $y_r\leq1$ give $\alpha,\beta\geq0$ and $\gamma,\eta\leq1$; the inequalities
$\gamma,\eta\geq0$, from $y_r\geq0$, are redundant, as $0\leq\alpha\leq\min(\gamma,\eta)$. So, up to redundant inequalities, exactly the stated system remains, and it is strictly satisfied at $(\frac14,\frac14,\frac34,\frac34)$; hence $\dim E_0=4$.
\end{proof}

Let $u_0$ be the point of $E_0$ with
$(\alpha,\beta,\gamma,\eta)=(\frac14,\frac14,\frac34,\frac34)$. Let $K$ be a chain as in the definition of $L$, with $a\in K$. Then $b\notin K$, as $a$ and $b$ are incomparable, so $K\setminus\{a\}\subseteq I\setminus\{a,b\}$, and the linear form $\sum_{p\in K\setminus\{a\}}x_p$ depends only on $\zeta$. These forms, for all such $K$, make up the \emph{$a$-family}. Let $\lambda_a(\zeta)$ be their largest value at $\zeta$; the $b$-family and $\lambda_b$ are defined in the same way. All these forms vanish at
$\zeta_0$, so $\lambda_a(\zeta_0)=\lambda_b(\zeta_0)=0$. Near $u_0$, each face $E\supseteq E_0$ of $B$ is the product of an open
set in the coordinates $(\alpha,\beta,\gamma,\eta)$ and a polyhedron in the $\zeta$-space; the next lemma makes this precise
and gives $L$ and $U$ near $u_0$.

\begin{lemma}\label{lem:product}
There are open neighborhoods $V_4$ of $(\frac14,\frac14,\frac34,\frac34)$ in $\mathbb R^4$ and $V_\zeta$ of $\zeta_0$ in the $\zeta$-space such that, for $N_0=V_4\times V_\zeta$, the following hold.
\begin{enumerate}
\item[(a)] For every face $E\supseteq E_0$ of $B$,
\begin{equation}\label{eq:product}
E\cap N_0=V_4\times(Q_E\cap V_\zeta),
\end{equation}
where $Q_E$ is the set of points of the $\zeta$-space that satisfy the defining inequalities of $B$ that are tight at $u_0$, with those that are tight on all of $E$ taken as equations (these inequalities involve only $\zeta$, and $Q_E$ is a polyhedron that contains $\zeta_0$); so
$\relint(E)\cap N_0=V_4\times(\relint(Q_E)\cap V_\zeta)$.
\item[(b)] On $B\cap N_0$ we have
$L=\max\bigl(\alpha+\lambda_a(\zeta),\beta+\lambda_b(\zeta)\bigr)$ and
$U=\min(\gamma,\eta)$.
\end{enumerate}
\end{lemma}
\begin{proof}
(a) Every defining inequality of $B$ that involves $\alpha$, $\beta$, $\gamma$ or $\eta$ is strict at $u_0$. Indeed, at $u_0$ every chain sum is at most $\frac14$ (a chain
contains at most one of $a$ and $b$, and all other $x_p$ vanish), while every $y_r$ equals $\frac34$ or $1$; so the inequalities $\alpha\geq0$, $\beta\geq0$, $0\leq\gamma\leq1$, $0\leq\eta\leq1$ and all chain inequalities that involve any of these four coordinates are strict at $u_0$. An order inequality that involves $\gamma$ has the form $\gamma\leq y_{r'}$ with
$g'<r'$, since no $r\in J'$ satisfies $r<g'$ (each $r\in J'$ satisfies $r>q$, and $g'$ covers $q$); and $r'\neq h'$, as $g'$ and $h'$
are incomparable, so $y_{r'}=1>\frac34$ at $u_0$. The same holds for $\eta$. So the inequalities
that are tight at $u_0$ involve only $\zeta$. A face $E\supseteq E_0$ is cut out by the defining inequalities of $B$ that are tight on $E$ (Section~\ref{sec:defs}), and these are tight at $u_0$.
Take $V_4$ and $V_\zeta$ so small that all inequalities that are strict at $u_0$ remain strict on $N_0=V_4\times V_\zeta$; this choice does not depend on $E$. Then a point of $N_0$ lies in $E$ if and only if it
satisfies the inequalities that are tight at $u_0$, with those tight on $E$ taken as equations. As these
involve only $\zeta$, this is \eqref{eq:product}. Every point of $E$ satisfies the inequalities that
define $Q_E$, so $E\subseteq\mathbb R^4\times Q_E$. Since $Q_E$ is convex and $V_\zeta$ is a neighborhood of its point $\zeta_0$, every segment from $\zeta_0$ to a point of $Q_E$ has an initial subsegment in $Q_E\cap V_\zeta$, so $\aff(Q_E\cap V_\zeta)=\aff(Q_E)$. Together with \eqref{eq:product}, this gives $\aff(E)=\mathbb R^4\times\aff(Q_E)$. As $N_0$ is open, $\relint(E)\cap N_0$ is the interior of $E\cap N_0$
relative to $\aff(E)$, which by \eqref{eq:product} is $V_4\times(\relint(Q_E)\cap V_\zeta)$.

(b) At $u_0$ a chain has sum $\frac14$ if it contains $a$ or $b$, and sum $0$ otherwise, the empty chain included. The terms of $U$ other than $\gamma$
and $\eta$ equal $1$, which exceeds $\frac34$. By continuity, after shrinking $N_0$ (which keeps (a) true), on all of $N_0$ every chain that contains $a$ or $b$ has a larger sum than every chain that contains neither, and $\gamma$ and $\eta$ are smaller than the other terms of $U$. Hence, on $B\cap N_0$, only chains containing $a$ or $b$ attain $L$, and only $\gamma$ or $\eta$ attains $U$. For a chain $K$ that contains $a$ we have
$\ell_K=\alpha+\sum_{p\in K\setminus\{a\}}x_p$, so the maximum over these chains is $\alpha+\lambda_a(\zeta)$;
likewise for $b$.
\end{proof}

\begin{lemma}\label{lem:coface-crossing}
Let $E\supseteq E_0$ be a face of $B$. Then there are $w\in\relint(E)$, a set $N$, cells $S_1,S_2,T_1,T_2$
and hyperplanes $H_\SSS,H_\TT$ that satisfy the hypotheses of Lemma~\ref{lem:crossing}.
\end{lemma}
\begin{proof}
Let $e=\dim E$, and take $N_0$ from Lemma~\ref{lem:product}. For $\zeta$ in the $\zeta$-space let
$\varepsilon(\zeta)=\lambda_b(\zeta)-\lambda_a(\zeta)$, and let $w(\zeta)$ be the point with coordinates $\zeta$
and $(\alpha,\beta,\gamma,\eta)=(\frac14+\frac{\varepsilon(\zeta)}2,\frac14-\frac{\varepsilon(\zeta)}2,\frac34,\frac34)$.
As $\lambda_a$ and $\lambda_b$ are continuous and vanish at $\zeta_0$, we have $w(\zeta_0)=u_0$, and $w(\zeta)\in N_0$
for all $\zeta$ near $\zeta_0$. As $\zeta_0\in Q_E$ and the nonempty polyhedron $Q_E$ is the closure of its relative interior, we can choose $\zeta_1\in\relint(Q_E)$ with $w(\zeta_1)\in N_0$ such that any two forms of the same family whose restrictions to $\aff(Q_E)$ differ take different values at $\zeta_1$ (this excludes only finitely many proper affine subspaces of $\aff(Q_E)$). Within each family, the forms of maximal value at $\zeta_1$ then have the same restriction to $\aff(Q_E)$, and the other forms remain smaller near $\zeta_1$. So if $\mu_a$ and $\mu_b$ are forms of maximal value at $\zeta_1$ in the $a$-family and in the $b$-family, respectively, then $\lambda_a=\mu_a$ and $\lambda_b=\mu_b$ on $Q_E$ near $\zeta_1$. Let $w=w(\zeta_1)$. Then $w\in N_0$ and $\zeta_1\in\relint(Q_E)$, so $w\in\relint(E)$ by Lemma~\ref{lem:product}(a). At $w$ the two terms of the maximum in Lemma~\ref{lem:product}(b) are both equal to $\frac14+\frac12(\lambda_a(\zeta_1)+\lambda_b(\zeta_1))$, and the two terms of the minimum are both equal to $\frac34$.

Let $N\subseteq E\cap N_0$ be an open ball in $\aff(E)$ around $w$, so small that $\lambda_a=\mu_a$
and $\lambda_b=\mu_b$ at the $\zeta$-coordinates of its points. By
Lemma~\ref{lem:product}(b), on $N$ the function $L$ is the maximum of
the affine functions $\phi_1=\alpha+\mu_a$ and $\phi_2=\beta+\mu_b$, and
$U=\min(\gamma,\eta)$. By \eqref{eq:product}, $E$ contains a segment
through $w$ in the $\alpha$-direction, so $\phi_1\not\equiv\phi_2$ on
$\aff(E)$. Let $N_1$ and $N_2$ be the intersections of $N$ with the two
closed half-spaces of $\aff(E)$ bounded by $H_\SSS=\{\phi_1=\phi_2\}\cap\aff(E)$, a hyperplane of $\aff(E)$ through $w$, as $\phi_1(w)=\phi_2(w)$. Since $L$ is affine on each $N_i$,
Lemma~\ref{lem:restrict} shows that $N_i$ lies in an $e$-cell $S_i$ of $\SSS|_E$. Let $C$ be an $e$-cell of $\SSS|_E$ that meets $N$. By Lemma~\ref{lem:restrict}, $L$ is affine on $C$; and $C\cap N$ is an $e$-dimensional convex set, as $N$ is open in $\aff(E)$ and $C$ is the closure of its relative interior. We claim that
$C\cap N$ lies in $N_1$ or in $N_2$. Otherwise $C\cap N$ has points strictly on both sides of $H_\SSS$; as $C$ is the closure of its relative interior, so does the set $\relint(C)\cap N$, which is open in $\aff(E)$. So $C\cap N$ contains a nonempty open set on which $\phi_1>\phi_2$ and one on which $\phi_2>\phi_1$. The affine function $L|_{C\cap N}$ equals $\phi_1$ on the first and $\phi_2$ on the second, so $\phi_1=\phi_2$ on $\aff(E)$, a contradiction. So the nonempty set $\relint(C)\cap N$, which is open in $\aff(E)$, lies in $N_i\subseteq S_i$ for some $i$, and hence in $\relint(S_i)$; as distinct cells of $\SSS$ have disjoint relative interiors, $C=S_i$. Thus $S_1$ and $S_2$ are the only $e$-cells of $\SSS|_E$
that meet $N$. Also $S_i\cap N=N_i$: the set $S_i\cap N$ contains $N_i$ and, by the argument just given with
$C=S_i$, lies in $N_1$ or in $N_2$, while $N_i\not\subseteq N_{3-i}$, as $\phi_1-\phi_2$ takes both signs on $N$. As $N_1\neq N_2$, it follows that $S_1\neq S_2$. The same holds for $\TT|_E$ and
$H_\TT=\{\gamma=\eta\}\cap\aff(E)$, as $E$ also contains a segment
through $w$ in the $\gamma$-direction. Finally, $H_\SSS\neq H_\TT$: moving from $w$ in the $\alpha$-direction, which is parallel to $\aff(E)$ by \eqref{eq:product}, changes $\phi_1-\phi_2$ but not
$\gamma-\eta$.
\end{proof}

\subsection{Faces containing \texorpdfstring{$E_0$}{E0}}
\begin{lemma}\label{lem:coface-count}
Let $E$ be a nonempty $e$-face of an $m$-dimensional polytope $Q$. For
$0\leq j\leq m-e$, the face $E$ lies in at least $\binom{m-e}{j}$ faces
of $Q$ of dimension $e+j$.
\end{lemma}
\begin{proof}
First, an $s$-dimensional polytope has at least as many $\kappa$-faces
as the $s$-simplex, namely $\binom{s+1}{\kappa+1}$ for $-1\leq\kappa\leq s$, the empty face being counted for $\kappa=-1$. We use induction on $s$. For $s\leq0$ the polytope is empty or a point, and the bound holds with equality. Let $s\geq1$. The polytope has a facet $D$ and, being the convex hull of its vertices, a vertex $v$ outside the proper face $D$. By the induction hypothesis for the $(s-1)$-polytope $D$, at least $\binom{s}{\kappa+1}$ of the $\kappa$-faces lie in $D$. The vertex figure of $v$ is an $(s-1)$-polytope whose $(\kappa-1)$-faces correspond to the $\kappa$-faces that contain $v$ \cite[Proposition~2.4]{Ziegler}, so, by the induction hypothesis, at least $\binom{s}{\kappa}$ of the $\kappa$-faces contain $v$. No face lies in $D$ and contains $v$, so Pascal's rule gives the bound.

Second, the faces of $Q$ containing $E$ correspond to those of a
polytope of dimension $m-e-1$, the face figure of $E$
\cite[Chapter~2]{Ziegler} (empty if $m=e$); an $(e+j)$-face corresponds to
a $(j-1)$-face (the empty face for $j=0$).
The first part ($s=m-e-1$, $\kappa=j-1$) gives at least
$\binom{m-e}{j}$ of them.
\end{proof}

\subsection{Proof of the inequality in Theorem~\ref{thm:main}}\label{subsec:sum}
\begin{proof}[Proof]
As $P$ contains $X$, we have $n\geq5$. For the switch of Section~\ref{subsec:E0}, $\dim B=n-1$ and $\dim
E_0=4$, so Lemma~\ref{lem:coface-count} gives at least
$\binom{n-5}{e-4}$ faces $E\supseteq E_0$ of each dimension $e$ with
$4\leq e\leq n-1$. Each such face has $\delta_E(t)\geq t^{e-2}(1+t)^2$ by
Lemmas~\ref{lem:coface-crossing} and~\ref{lem:crossing}, and every other face has $\delta_E(t)\geq0$ by Lemma~\ref{lem:carrier}(iii) and
Proposition~\ref{prop:transverse}. By Lemma~\ref{lem:identity},
\[
F_{\OO_{I\cup\{q\}}(P)}-F_{\OO_I(P)}
\geq\sum_{e=4}^{n-1}\binom{n-5}{e-4}t^{e-2}(1+t)^2
=t^2(1+t)^2(1+t)^{n-5}=t^2(1+t)^{n-3}.
\]
Now take a linear extension of $I$, followed by $q$ and then by a linear extension of $J'$; as $I$ is an order ideal and $q=\min J$,
this is a linear extension $p_1,\ldots,p_n$ of $P$. Each
$I_j=\{p_1,\ldots,p_j\}$ ($0\leq j\leq n$, $I_0=\varnothing$) is an order ideal and $p_{j+1}$ is minimal in $P\setminus I_j$ for $j<n$; with $j_0=|I|$ we have
$I_{j_0}=I$ and $I_{j_0+1}=I\cup\{q\}$. By
Proposition~\ref{prop:mono},
\[
F_{\OO(P)}=F_{\OO_{I_0}(P)}\leq F_{\OO_{I_1}(P)}\leq\cdots\leq
F_{\OO_{I_n}(P)}=F_{\CC(P)},
\]
and the difference $F_{\CC(P)}-F_{\OO(P)}$ is the sum of the changes over the $n$ steps, each of them nonnegative; the change
from $I$ to $I\cup\{q\}$ is at least $t^2(1+t)^{n-3}$, which proves the inequality.
\end{proof}

\section{Sharpness and the equality case}\label{sec:sharp}
\begin{proposition}\label{prop:X} We have $F_{\CC(X)}-F_{\OO(X)}=t^2+2t^3+t^4$.
\end{proposition}
\begin{proof}
We switch the elements of $X$ in the order $a,b,c,g,h$ and compute each
change with Lemma~\ref{lem:identity}. At the switches of $a$ and $b$, no
element of $I$ lies below the switched element, so $L\equiv0$ is affine on $B$. Hence $\SSS$ consists of the nonempty faces of $B$, so $F_\SSS=F_B$; and $\RR=\TT$, because each cell $G\cap T$ ($G$ a face of $B$, $T\in\TT$) is a
face of $T$. So the change is $F_B+F_\TT-F_B-F_\TT=0$; likewise at the switches of $g$ and $h$, where $U\equiv1$.

At the switch of $c$ we have $I=\{a,b\}$ and $J'=\{g,h\}$. This is the switch of Section~\ref{subsec:E0} for
$P=X$: there $q=c$, as the strict upper sets of $g$ and $h$ are empty, hence chains. Here $g'=g$ and $h'=h$, the vector $\zeta$ is empty, and $B=E_0$, with $L=\max(\alpha,\beta)$ and $U=\min(\gamma,\eta)$.
The base $B$ is the set of solutions of the eight inequalities $\alpha\geq0$, $\beta\geq0$, $\gamma\leq1$, $\eta\leq1$,
$\alpha\leq\gamma$, $\alpha\leq\eta$, $\beta\leq\gamma$, $\beta\leq\eta$ (Lemma~\ref{lem:E0}), so every nonempty proper face of $B$ lies in one of the hyperplanes $\alpha=0$, $\beta=0$, $\gamma=1$, $\eta=1$, $\alpha=\gamma$, $\alpha=\eta$, $\beta=\gamma$, $\beta=\eta$, as one of the eight is tight on it (Section~\ref{sec:defs}). On the first four, $L$ or $U$ is affine (for example, $L=\beta$ on $\alpha=0$).
On the other four, $L$ and $U$ are both affine: for example, on $\alpha=\gamma$ we have
$\beta\leq\gamma=\alpha\leq\eta$, so $L=\alpha$ and $U=\gamma$ there; the other three are symmetric. Thus, for every nonempty proper face $E$ of $B$, the function $L$ or $U$ is affine on
$E$. By symmetry, suppose that $L$ is affine on $E$, and let $e=\dim E$. By Lemma~\ref{lem:restrict}, $E$ lies
in an $e$-cell of $\SSS|_E$, which is therefore $E$ itself; so $E$ is a cell of $\SSS$. Every cell of $\SSS|_E$
is a cell of $\SSS$ contained in $E$, and it meets the cell $E$ in a common face, so it is a face of $E$.
Hence $\SSS|_E$ consists of the nonempty faces of $E$. Then $\RR|_E=\TT|_E$, since
each $G\cap T$ ($G$ a face of $E$, $T\in\TT|_E$) is a face of $T$. As a cell with carrier $E$ lies in $E$, the cells of $\RR$ and of $\TT$ with carrier $E$ are the same, and $E$ is the only cell of $\SSS$ with carrier $E$. Hence $\delta_{E,k}=\rho_E(k)-[k=e]-\tau_E(k)+[k=e]=0$ for all $k$, and only the face $B$ contributes to the sum in Lemma~\ref{lem:identity}.

The maximal cells of $\SSS$ are $B\cap\{\alpha\geq\beta\}$ and
$B\cap\{\alpha\leq\beta\}$, and those of $\TT$ are $B\cap\{\gamma\geq\eta\}$ and $B\cap\{\gamma\leq\eta\}$. Indeed, by the description of the maximal cells in Section~\ref{sec:cited}, these four sets are the full-dimensional sets $A_\ell$: here $L=\max(0,\alpha,\beta)$, and
$A_0=B\cap\{\alpha=\beta=0\}$ is not full-dimensional; likewise, $U=\min(\gamma,\eta,1)$, and $A_1=B\cap\{\gamma=\eta=1\}$ is not full-dimensional either. The cells with carrier $B$ are as follows: in $\SSS$, two $4$-cells and the $3$-cell $B\cap\{\alpha=\beta\}$; in $\TT$, two
$4$-cells and the $3$-cell $B\cap\{\gamma=\eta\}$; in $\RR$, four
$4$-cells, four $3$-cells and the $2$-cell
$B\cap\{\alpha=\beta,\ \gamma=\eta\}$. Indeed, a face of $B\cap\{\alpha\geq\beta\}$ is cut out by some of the eight inequalities of $B$ and the inequality $\alpha\geq\beta$,
and if one of the eight is tight on such a face, the face lies in a proper face of $B$; so the cells of $\SSS$ with carrier
$B$ are $B\cap\{\alpha\geq\beta\}$, $B\cap\{\alpha\leq\beta\}$ and their common face $B\cap\{\alpha=\beta\}$, and
likewise for $\TT$. A cell $S\cap T$ of $\RR$ has carrier $B$ only if $S$ and $T$ do. The nine intersections of
these three cells of $\SSS$ with these three cells of $\TT$ each contain $u_0=(\frac14,\frac14,\frac34,\frac34)$,
an interior point of $B$, so each has carrier $B$; near $u_0$ they are the sets where $\alpha-\beta$ and
$\gamma-\eta$ have prescribed signs or vanish, as in Figure~\ref{fig:slice}, so they are distinct and have the
stated dimensions. So
$\delta_B=(1-0-0)t^2+(4-1-1)t^3+(4-2-2+1)t^4=t^2+2t^3+t^4$; adding the
five changes gives the claim.
\end{proof}

\begin{proposition}\label{prop:sharp}
Let $P_5=X$, and, for $m\geq5$, let $P_{m+1}$ be obtained from $P_m$ by adjoining a new least element $\hat0$. Then, for every
$n\geq5$, the poset $P_n$ has $n$ elements, contains $X$, and satisfies
$F_{\CC(P_n)}-F_{\OO(P_n)}=t^2(1+t)^{n-3}$.
\end{proposition}
\begin{proof}
Each step adds one element, so $|P_n|=n$, and $P_n$ contains $P_5=X$ as an induced subposet. For
$0\leq\theta\leq1$, the points of $\OO(P_{m+1})$ with $v_{\hat0}=\theta$ are the points $(\theta,v)$ with
$v\in[\theta,1]^{P_m}$ and $v_p\leq v_{p'}$ for $p\leq p'$ in $P_m$ (the inequalities $v_{\hat0}\leq v_p$ become $v_p\geq\theta$), that is, the points $(\theta,\theta\mathbf1+(1-\theta)v')$ with $v'\in\OO(P_m)$, where
$\mathbf1\in\mathbb R^{P_m}$ is the all-ones vector. We have $(\theta,\theta\mathbf1+(1-\theta)v')=(1-\theta)(0,v')+\theta\mathbf1'$, where $\mathbf1'$ is the all-ones vector of $\mathbb R^{P_{m+1}}$. As $\mathbf1'$ does not lie on the hyperplane $v_{\hat0}=0$, the polytope $\OO(P_{m+1})$ is the pyramid over $\{0\}\times\OO(P_m)$ with apex $\mathbf1'$. Every maximal chain of $P_{m+1}$ contains $\hat0$. As all coordinates are nonnegative on $\CC(P_{m+1})$, the chain inequalities of the maximal chains imply the others, so the set of points of $\CC(P_{m+1})$ with $x_{\hat0}=\theta$ is $\{\theta\}\times(1-\theta)\CC(P_m)$. Since $(\theta,(1-\theta)x)=(1-\theta)(0,x)+\theta\mathbf e_{\hat0}$ for $x\in\CC(P_m)$, the polytope $\CC(P_{m+1})$ is the pyramid over $\{0\}\times\CC(P_m)$ with apex the unit vector $\mathbf e_{\hat0}$.
By the formula $F_{\pyr(Q)}=1+(1+t)F_Q$ of Section~\ref{sec:defs}, each step multiplies $F_{\CC(P_m)}-F_{\OO(P_m)}$ by
$1+t$ (the constant $1$ in $1+(1+t)F_Q$ cancels), so Proposition~\ref{prop:X} gives
$F_{\CC(P_n)}-F_{\OO(P_n)}=(1+t)^{n-5}(t^2+2t^3+t^4)=t^2(1+t)^{n-3}$.
\end{proof}
A characterization of the posets that attain the bound of Theorem~\ref{thm:main} is given in the expanded
version \cite[Theorem~6.1]{G}; it is not used here.

\begin{corollary}\label{cor:equality}
For an $n$-element poset $P$ and $2\leq k\leq n-1$, the following are
equivalent: (i)~$f_k(\OO(P))=f_k(\CC(P))$; (ii)~$P$
does not contain $X$; (iii)~$\OO(P)$ and $\CC(P)$ are unimodularly
equivalent.
\end{corollary}
\begin{proof}
By Theorem~\ref{thm:HL}, (ii) implies (iii), and
(iii) implies (i) because unimodularly equivalent polytopes have
the same $f$-vector (Section~\ref{sec:defs}). If $P$ contains $X$, the
coefficient of $t^k$ in $t^2(1+t)^{n-3}$ is $\binom{n-3}{k-2}$, which is at least $1$, so
$f_k(\CC(P))>f_k(\OO(P))$ by Theorem~\ref{thm:main}; hence (i) implies (ii).
\end{proof}

\begin{proof}[Proof of Theorem~\ref{thm:resolution}]
(a) This is Fourier's Theorem~1.1 \cite{Fourier}. It also follows from Proposition~\ref{prop:mono}: for any linear
extension $p_1,\ldots,p_n$ of $P$, the sequence of switches of $p_1,\ldots,p_n$, as in Section~\ref{subsec:sum},
gives
$F_{\OO(P)}\leq F_{\CC(P)}$; this is Fourier's argument. (b) This is the equivalence of (i) and (iii) in
Corollary~\ref{cor:equality}. (c) The numbers of edges agree by a theorem of Hibi, Li, Sahara and Shikama \cite[Theorem~2.1]{HLSS}; if $P$ contains
$X$, then $\OO(P)$ and $\CC(P)$ are not unimodularly equivalent by Theorem~\ref{thm:HL}.
\end{proof}

\paragraph{Acknowledgments.}
This work was done overwhelmingly by AI models and supervised by the
author. The models were Anthropic's Claude and OpenAI's ChatGPT. The
author takes full responsibility for the paper.

\begin{footnotesize}

\end{footnotesize}
\end{document}